\documentclass[]{interact}

\usepackage{epstopdf}% To incorporate .eps illustrations using PDFLaTeX, etc.
\usepackage[caption=false]{subfig}% Support for small, `sub' figures and tables

\usepackage[numbers,sort&compress]{natbib}% Citation support using natbib.sty
\usepackage{algorithmicx,algorithm}
\usepackage{multirow}
\bibpunct[, ]{[}{]}{,}{n}{,}{,}% Citation support using natbib.sty
\makeatletter% @ becomes a letter
\def\NAT@def@citea{\def\@citea{\NAT@separator}}% Suppress spaces between citations using natbib.sty
\makeatother% @ becomes a symbol again

\theoremstyle{plain}% Theorem-like structures provided by amsthm.sty
\newtheorem{theorem}{Theorem}[section]
\newtheorem{lemma}[theorem]{Lemma}

\theoremstyle{definition}
\newtheorem{definition}[theorem]{Definition}
\newtheorem{example}[theorem]{Example}

\theoremstyle{remark}

\newtheorem{assumption}{Assumption}

\begin{document}

%\articletype{ARTICLE TEMPLATE}% Specify the article type or omit as appropriate

\title{Sparse least squares solutions of multilinear equations}

\author{
\name{Xin Li, Ziyan Luo \thanks{CONTACT Ziyan Luo. Email: zyluo@bjtu.edu.cn} and Yang Chen}
\affil{School of Mathematics and Statistics, Beijing Jiaotong University, Beijing, 100044, China}
}

\maketitle

\begin{abstract}
In this paper, we propose a sparse least squares (SLS) optimization model for solving multilinear equations, in which the sparsity constraint on the solutions can effectively reduce storage and computation costs. By employing variational properties of the sparsity set, along with differentiation properties of the objective function in the SLS model, the first-order optimality conditions are analyzed in terms of the stationary points. Based on the equivalent characterization of the stationary points, we propose the Newton Hard-Threshold Pursuit (NHTP) algorithm and establish its locally quadratic convergence under some regularity conditions. Numerical experiments conducted on simulated datasets including cases of Completely Positive(CP)-tensors and symmetric strong M-tensors illustrate the efficacy of our proposed NHTP method.
\end{abstract}

\begin{keywords}
Sparse least squares, multilinear equations, Newton Hard-Thresholding Pursuit, Completely Positive tensor, symmetric strong M-tensor
\end{keywords}

\section{Introduction}
\label{sec:Introduction}
Multilinear equations (also known as tensor equations) have a wide range of applications in engineering and scientific computing such as data mining, numerical partial differential equations, tensor complementarity problems and high-dimensional statistics \citep{yan2022homotopy}. A recent line of research has been focused on numerical algorithms for solving multilinear systems with various coefficient tensors, see, e.g., CP-tensors \citep{yan2022homotopy}, strong M-tensors \citep{ding2016solving,han2017homotopy,li2017splitting,xie2018tensor,cui2019preconditioned} and other structured tensors \citep{xie2017fast,li2015solving,liang2021alternating,li2019alternating,beik2021preconditioning}. Furthermore, as the data dimension grows in practical problems, the sparsity constraint on the solutions turns to be a reasonable choice to effectively alleviate the ``curse of dimensionality" in systems of multilinear equations.
\vskip 2mm

As the direct and accurate characterization of the entry-wise sparsity in vectors, the so-called $\ell_0$-norm (i.e., the number of nonzero entries in the vector) is nonconvex and discontinuous. Thus, the problem of finding sparse solutions of multilinear equations in the sense of least squares is generally NP-hard in computational complexity. Little can be found in this research direction, except the work \citep{luo2017sparsest} in which the strong $M$-tensor coefficient and non-negative right-hand side vector are considered. In this special setting, they showed that the sparest solution to the corresponding multilinear equations can be obtained by solving the $\ell_1$-norm (i.e., the sum of absolute values of all entries) relaxation problem. However, for general multilinear equations with sparsity constraint, little  can be found to our best knowledge. This motivates us to consider the sparse least squares (SLS) optimization model for solving multilinear equations.
\vskip 2mm

The SLS model is actually a special case of cardinality constrained optimization (CCO) problems. Existing algorithms for CCO can be roughly classified into two categories. The first category is the ``relaxation" method by using continuous and/or convex surrogates of the involved $\ell_0$-norm, see. e.g., \citep{nikolova2013description,shen2013constrained,bertsimas2009algorithm,gotoh2018dc}. The other category is the ``greedy" type method by tackling the involved $\ell_0$-norm directly. Typical algorithms include the matching pursuit algorithm \citep{2008Tropp}, the iterative projection algorithm \citep{beck2013sparsity,2011Recipes,2010Normalized,2012Accelerated}, the hard threshold pursuit algorithm \citep{foucart2011hard}, just name a few. Particularly, some second-order methods are proposed. For example, Yuan et al. \citep{yuan2017gradient} and Bahmani et al. \citep{2013Greedy} realized that restricted Newton steps can be employed in the underlying subproblems in subspaces to improve the algorithm performance. Recently, Zhou et al. \citep{zhou2021global} proposed a Newton Hard Thresholding Pursuit (NHTP) algorithm with quadratic convergence rate under some regularity assumptions. Its superior numerical performance and theoretical convergence inspire us to develop NHTP for solving the SLS model.
\vskip 2mm

The main contributions of this paper are three-fold. First, the SLS model is proposed to solve mulinear equations with the cardinality constraint. Second, the optimality of the SLS model, and the regularity properties of the objective functions are elaborated, which serve as the crucial theoretical guarantees for designing the Newton-type algorithm in the sequel. Third, the NHTP is developed for solving SLS and its locally quadratic convergence is established.

The remainder of the paper is organized as follows. In Section \ref{sec:Preliminaries}, we review the related tensor basics and notations; In Section \ref{sec:Main Results}, we analyze the first-order optimality conditions for the proposed SLS optimization model and regularity properties of the objective function in theory, and then design a quadratically convergent NHTP in algorithm. In Section \ref{sec:Numerical Experiments}, we report some numerical results to verify the proposed NHTP algorithm, by comparing with the existing homotopy algorithm. Concluding remarks are drawn in Section \ref{sec:Concluding remarks}.

\section{Preliminaries}
\label{sec:Preliminaries}
% Firstly, the symbols and concepts of tensors and tensor operations used in this paper are briefly described. Vectors are represented in italic lowercase letters, for example, $x$; Matrices are represented in capital letters, for example, $\mathrm{I}$, identity matrices of appropriate size; Tensors are represented by Fraktur letters, for example, $\mathcal{A} $. All tensors of $N$-order $I_1\times I_2\times \cdots \times I_N$-dimension over real fields are represented by $\mathbb{R} ^{I_1\times I_2\times \cdots \times I_N}$.

For any positive integer, denote $[n]:=\{1,2,\dots,n\}$, $\mathbb{R} _{+}^{n}:=\{ x\in \mathbb{R} ^n\,\,: x\ge 0 \} $.
% Given positive integer $I_1,\cdots ,I_N$, the $N$-order tensor represents a high-order high-dimensional array, and its component is $x_{i_1i_2\cdots i_N},   i_n\in [I_n],\, n=1,...,N.$
We denote $\mathcal{A}=(a_{i_{1} i_{2} \cdots i_{m}})$ as an $m$-th order $n$-dimensional tensor with $i_j\in [n]$, $j\in [m]$, and then the linear space of all $m$-th order $n$-dimensional tensors as $\mathbb{R} ^{[m,n]}:=\mathbb{R} ^{n\times n\times \cdots \times n}$.  For any $\mathcal{A} \in \mathbb{R}^{[m,n]}$, we call $\mathcal{A}$ is a symmetric tensor if its entries remain unchanged under any permutation of the indices and denote the set of all symmetric tensors in $\mathbb{R} ^{[m,n]}$ as $S^{[m, n]}$. If there exists a positive integer $r$ and $u^{(k)} \in \mathbb{R}_+^n$ such that $\mathcal{A} =\sum\nolimits_{k=1}^r{( u^{( k )} ) ^m}$, where $( u^{(k)} ) ^m=( u_{i_1}^{(k)}\cdots u_{i_m}^{(k)} ) $, then $\mathcal{A} $ is called as a Completely Positive tensor (CP-tensor). Furthermore, if $\mathrm{Span}\{ u^{(1)},u^{(2)},...,u^{(r)} \} =\mathbb{R} ^n$, then $\mathcal{A}$ is called as a Strong Completely Positive tensor (SCP-tensor). We denote $\mathrm{CP}^{[ m,n ]}$ and $\mathrm{SCP}^{[ m,n ]}$ as the sets of all $m$-th order $n$-dimensional CP-tensors and SCP-tensors, respectively.

% An $N$-order tensor $\mathcal{X} \in \mathbb{R} ^{I_1\times I_2\times \cdots \times I_N}$ is called a rank-one tensor if it can be expressed as the outer product of $N$ vectors, i.e.,
% $$ \mathcal{X} =u^{(1)}\circ u^{(2)}\circ \cdots \circ u^{(N)}, $$
% Where, $\circ $ represents the outer product between vectors, $u^{(n)}\in \mathbb{R} ^{I_n},  n=1,...,N$. At this time, the component of $\mathcal{X} $ is:
% $$ x_{i_1i_2\cdots i_N}=u_{i_1}^{(1)}u_{i_2}^{(2)}\cdots u_{i_N}^{(N)},  1\le i_n\le I_n .$$

% The n-mode vector product can be defined between tensor and vector. For any tensor
% $\mathcal{X} \in \mathbb{R} ^{I_1\times I_2\times \cdots \times I_N}$ and any vector $v\in \mathbb{R} ^n$, the n-mode vector product of the two is a tensor of order N-1, the size is $I_1\times ,\cdots ,I_{n-1}\times I_{n+1}\times \cdots \times I_N$, denoted by $\mathcal{X} \cdot _nv$, whose component element is
% $$( \mathcal{X} \cdot _nv ) _{i_1,...,i_{n-1}i_{n+1},...,i_N}=\sum\nolimits_{i_n=1}^{I_n}{x_{i_1,...,i_n,...,i_N}}v_{i_n}. $$
Given $\mathcal{A} \in \mathbb{R} ^{[m,n]}$ and $b\in \mathbb{R} ^n$, multilinear equations can be expressed as $\mathcal{A} x^{m-1}=b$ with
$(\mathcal{A} x^{m-1} ) _i=\sum\nolimits_{i_2,...,i_m=1}^n{a_{ii_2\cdots i_m}}x_{i_2}\cdots x_{i_m}$ for $i\in [n]$.
For any $ d\in [m-1]$, $\mathcal{A} x^{m-d}\in \mathbb{R} ^{[d,n]}$ with entries $ (\mathcal{A} x^{m-d} ) _{i_1\cdots i_d}=\sum\nolimits_{i_{d+1},...,i_m=1}^n{a_{i_1\cdots i_di_{d+1}\cdots i_m}}x_{i_{d+1}}\cdots x_{i_m}$ for $i_j\in [n]$, $j\in [d]$. We define $\lambda $ is an eigenvalue of $\mathcal{A}$ and $x$ is a corresponding eigenvector if there exists $x\in \mathbb{R} ^n\backslash \{0\}$ and $\lambda \in \mathbb{R} $ such that
$\mathcal{A} x^{m-1}=\lambda x^{[m-1]}$ with $x^{[m-1]}:=[ x_{1}^{m-1},x_{2}^{m-1},...,x_{n}^{m-1} ] ^{\top}\in \mathbb{R}^n$. Then the spectral radius of $\mathcal{A} $ is defined by
$\rho(\mathcal{A})=\max \{|\lambda|: \mathcal{A} x^{m-1}=\lambda x^{[m-1]}, x>0\}$.
Let $\mathcal{I} \in \mathbb{R}^{[m,n]}$ be an identity tensor (i.e., the diagonal entries are 1, and the other entries are 0), a tensor $\mathcal{A}$ is called as an M-tensor if there exists a nonnegative tensor $\mathcal{B}$ and a positive real number $\mathrm{s}\ge \rho (\mathcal{B} )$ such that $\mathcal{A} =\mathrm{s}\mathcal{I} -\mathcal{B} $; if $\mathrm{s}>\rho (\mathcal{B} )$, $\mathcal{A}$ is a strong M-tensor. For convenience, notations that will be used in the paper are listed in Table \ref{table1}.

\begin{table}%[!htp]
%\scriptsize
%\captionsetup{font={scriptsize}}
\tbl{A list of notation. \label{table1}}
{\begin{tabular}{ll}
\toprule
$\mathrm{supp}( x )$  & $:=\{i\in[n]: x_i\neq 0\}$ the support set of $x$.\\
$\|x\|_0$   & $:= \sharp\{i\in[n]: x_i\neq 0\}$ is the $l_0$-norm of $x$\\
$x_{( i )}$  & the $i$-th largest element of $x$.\\
$\Gamma ^*$  & $:=\mathrm{supp}( x^* ) $\\
$T$  & index set from $\{1,2,...,n\}$.\\
$|T|$  & cardinality of $T$.\\
$T^c$  & the complementary set of $T$.\\
$x_T$  & the sub vector of $x$ containing elements indexed on $T$.\\
$\nabla _Tf(x)$  & $:=(\nabla f(x))_T$\\
$\nabla _{T,J}^{2}f(x)$  & the submatrix of the Hessian matrix with rows and columns are indexed by $T$ and $J$ respectively.\\
$\nabla _{T}^{2}f(x)$  & $:=\nabla _{T,T}^{2}f(x)$\\
$\nabla _{T,\cdot}^{2}f(x)$  & the submatrix of the Hessian matrix with rows are indexed by $T$.\\
$\| x\| $  & the Euclidean norm of the vector $x$.\\
$\| \mathrm{A} \| $  & the  spectral norm of the matrix $\mathrm{A}$ (i.e.,  the maximum singular value of the  matrix $\mathrm{A}$).\\
$\| \mathcal{A} \| _F$  &  the F norm of the tensor $\mathcal{A}$.  If $\mathcal{A} \in \mathbb{R} ^{[m,n]}$, then $\| \mathcal{A} \| _F:=( \sum\nolimits_{i_1,i_2,...,i_m=1}^n{| a_{i_1i_2...i_m} |^2} ) ^{\frac{1}{2}}$.\\
$\mathcal{J} _s(x)$  & $:=\{ J\subseteq [ n ] \,: \, | J |=s, \mathrm{supp(}x)\subseteq J \} $.\\
$Q_{2s}(x)$  & $:=\{ T\subseteq [ n ] \,: \, | T |\leqslant 2s, \mathrm{supp(}x)\subseteq T \} $.\\
\bottomrule
\end{tabular}}
\end{table}

\section{Main results}
\label{sec:Main Results}
In this section, we consider the following sparse least squares (SLS) optimization problem of multilinear equations:
\begin{equation}\label{eq:model}
    \underset{x\in \mathbb{R} ^n}{\min}\,\,f(x) :=\frac{1}{2}\| \mathcal{A} x^{m-1}-b \| ^2, \quad   \mathrm{s.t.} \ \| x \| _0\leqslant s,
\end{equation}
where $\mathcal{A} \in S^{[m,n]}$, the integer $s\in[1,n)$ is a prescribed upper bound that controls the sparsity of $x$. Note that $f$ is differentiable of any order. By virtue of the symmetry of the tensor $\mathcal{A}$, the gradient and the Hessian matrix of $f$ at any $x\in {\mathbb{R}}^n$ take the forms of
\begin{align}
&\nabla f( x ) =( m-1 ) \mathcal{A} x^{m-2}( \mathcal{A} x^{m-1}-b ) ,\label{gradient}\\
&\nabla ^2f( x ) =( m-1 ) ( m-2 ) \mathcal{A} x^{m-3}( \mathcal{A} x^{m-1}-b ) +( m-1 ) ^{2}( \mathcal{A} x^{m-2} ) ^2.\label{hessian}
\end{align}

\subsection{Optimality analysis}
This subsection is dedicated to the discussion of several differential properties of the objective function $f$ and the optimality conditions for the problem (\ref{eq:model}), all of which will provide theoretical guarantees for the design of the Newton-type algorithm in next subsection. %We first give some definitions including restricted strongly smooth (RSS), restricted strongly convex (RSC) and locally restricted Hessian Lipschitz continuous \citep{zhou2021global}.
\begin{definition}(\citep{zhou2021global})
\label{definition1}
Suppose that $f:\mathbb{R} ^n\mapsto \mathbb{R}$ is a twice continuously differentiable function. Let $M_{2s}(x):=\mathop {\mathrm{sup}} \nolimits_{y\in \mathbb{R} ^n}\{\langle y,\nabla ^2f(x)y \rangle : |\mathrm{supp(}x)\cup \mathrm{supp(}y)| \le 2\mathrm{s},\| y\| =1 \}$ and $m_{2s}(x):=\mathop {\mathrm{inf}} \nolimits_{y\in \mathbb{R} ^n}\{ \langle y,\nabla ^2f(x)y \rangle \,: \, |\mathrm{supp(}x)\cup \mathrm{supp(}y)| \le 2\mathrm{s},\| y\| =1 \}$ for any $s$-sparse vector $x$.
\begin{itemize}
\item[(1)]We say $f$ is $M_{2s}-$RSS if there exists a constant $M_{2s} > 0$ such
that $M_{2s}(x)\le M_{2s}$.
% $f$
% is said to be locally $RSS$ at $x$ if $M_{2s}(x)\le M_{2s}$ only holds for those $s$-sparse vectors $z$
% in a neighborhood of $x$.
\item[(2)] We say $f$ is $m_{2s}-$RSC if there exists a constant $m_{2s} > 0$ such that $m_{2s}(x)\ge m_{2s}$.
% $f$ is
% said to be locally $RSC$ at $x$ if $m_{2s}(z) \ge m_{2s}$ only holds for those $s$-sparse vectors $z$ in
% a neighborhood of $x$.
\item[(3)] We say $f$ is locally restricted Hessian Lipschitz continuous at
$x$ if there exists a constant $L_f>0$ and a neighborhood $\mathcal{N} _s(x):=\{ z\in \mathbb{R} ^n \, :\, \mathrm{supp}(x)\subseteq \mathrm{supp}(z)$, $\| z\| _0\le \mathrm{s} \}$ such that
$$\| \nabla _{T,\cdot}^{2}f(y)-\nabla _{T,\cdot}^{2}f(z) \| \le L_f\| y-z\| ,\quad  \forall y,z\in \mathcal{N} _s(x)$$
for any index set $T$ satisfying $|T|\le s$ and $T\supseteq \mathrm{supp}(x)$.
\end{itemize}
\end{definition}

% \begin{remark}
% It is easy to verify by the definition \ref{definition1} that $f$ is $M_{2s}-RSS$ if and only if
% $$\|\nabla f(x)-\nabla f(y)\| \le M_{2s}\| x-y\| ,  \quad  \forall x,y,|\mathrm{supp(}\mathbf{x})|\le s,|\mathrm{supp(}x)\cup \mathrm{supp(}y)|\le 2s.$$
% \end{remark}

\begin{theorem}\label{thm:locally Hessian Lipschitz continuous}
Suppose that $x^*$ is an optimal solution of (\ref{eq:model}). There exists $\delta _0>0$ such that for any $x\in \mathcal{N} _s(x^*,\delta _0):=\{ x\in \mathbb{R} ^n\,\,: \mathrm{supp(}x^*)\subseteq \mathrm{supp(}x),\| x\| _0\le \mathrm{s},\| x-x^* \| <\delta _0 \}$, $f$ is locally Hessian Lipschitz continuous near $x^*$, and the Lipschitz constant is given by
\begin{equation}\label{equ:L_f}
\begin{aligned}
	L_f = & 2(m-1)(m-2)(2m-3)\| \mathcal{A} \| _{F}^{2}( \| x^* \| +\delta _0 ) ^{2m-5}\\
	      & +(m-1)(m-2)(m-3)\| b\| \| \mathcal{A} \| _{F}^{2}( \| x^* \| +\delta _0 ) ^{m-4}.
\end{aligned}
\end{equation}
\end{theorem}

\begin{proof}
For any $ x,y\in \mathcal{N} _s(x^*,\delta _0)$, we have $\| x \| \le \| x^* \| +\delta _0$, $ \| y \| \le \| y^* \| +\delta _0,$ then
$$
\begin{aligned}
\|\nabla^{2} f(x)-\nabla^{2} f(y)\|
=& \|(m-1)(m-2) \mathcal{A} x^{m-3}(\mathcal{A} x^{m-1}-b)+(m-1)^{2}(\mathcal{A} x^{m-2})^{2} \\
&-(m-1)(m-2) \mathcal{A} y^{m-3}(\mathcal{A} y^{m-1}-b)+(m-1)^{2}(\mathcal{A} y^{m-2})^{2} \| \\
\leq &(m-1)(m-2)\|\mathcal{A} x^{m-3}(\mathcal{A} x^{m-1}-b)-\mathcal{A} y^{m-3}(\mathcal{A} y^{m-1}-b)\| \\
&+(m-1)^{2}\|(\mathcal{A} x^{m-1})^{2}-(\mathcal{A} y^{m-2})^{2}\|.
\end{aligned}
$$
For simplicity, denote
$$
\begin{aligned}
	c1:= &\| \mathcal{A} x^{m-3}(\mathcal{A} x^{m-1}-b)-\mathcal{A} y^{m-3}(\mathcal{A} y^{m-1}-b) \|\\
	\le & \| \mathcal{A} x^{m-3}\mathcal{A} x^{m-1}-\mathcal{A} y^{m-3}\mathcal{A} y^{m-1} \| +\| \mathcal{A} x^{m-3}-\mathcal{A} y^{m-3} \| _F\| b \|\\
	\le & \| \mathcal{A} x^{m-3} \| _F\| \mathcal{A} x^{m-1}-\mathcal{A} y^{m-1} \| +\| \mathcal{A} y^{m-1} \| \| \mathcal{A} x^{m-3}-\mathcal{A} y^{m-3} \| _F\\
	&+\| \mathcal{A} x^{m-3}-\mathcal{A} y^{m-3} \| _F\| b \|.\\
    c2:= & \| (\mathcal{A} x^{m-2})^2-(\mathcal{A} y^{m-2})^2 \|\\
	= &\| (\mathcal{A} x^{m-2})  (\mathcal{A} x^{m-2})-(\mathcal{A} x^{m-2})(\mathcal{A} y^{m-2})+(\mathcal{A} x^{m-2})(\mathcal{A} y^{m-2})\\
	&-(\mathcal{A} y^{m-2}) (\mathcal{A} y^{m-2}) \|\\
	\le  & \| \mathcal{A} x^{m-2} \| \| \mathcal{A} x^{m-2}-\mathcal{A} y^{m-2} \| +\| \mathcal{A} y^{m-2} \| \| \mathcal{A} x^{m-2}-\mathcal{A} y^{m-2} \|.
\end{aligned}
$$
We can see that
\begin{align*}
\|\mathcal{A} x^{m-1}-\mathcal{A} y^{m-1}\|
=&\|\mathcal{A} x x \cdots x-\mathcal{A} x x \cdots y+\mathcal{A} x x \cdots y+\cdots+\mathcal{A} x y \cdots y-\mathcal{A} y \cdots y\| \\
\leq &\|\mathcal{A} x \cdots(x-y)\|+\|\mathcal{A} x \cdots(x-y) y\|+\cdots+\|\mathcal{A}(x-y) y \cdots y\| \\
\leq &(\|\mathcal{A}\|_{F}\|x\|^{m-2}+\|\mathcal{A}\|_{F}\|x\|^{m-3}\|y\|+\cdots+\|\mathcal{A}\|_{F}\|y\|^{m-2})\|x-y\| \\
\leq &(m-1)\|\mathcal{A}\|_{F}(\|x^{*}\|+\delta_{0})^{m-2}\|x-y\|.
\end{align*}
Similarly, one has
\begin{align*}
&\|\mathcal{A} x^{m-2}-\mathcal{A} y^{m-2}\| \leq(m-2)\|\mathcal{A}\|_{F}(\|x^{*}\|+\delta_{0})^{m-3}\|x-y\| ,\\
&\|\mathcal{A} x^{m-3}-\mathcal{A} y^{m-3}\| \leq(m-3)\|\mathcal{A}\|_{F}(\|x^{*}\|+\delta_{0})^{m-4}\|x-y\|.
\end{align*}
Plugging the above three inequalities into c1 and c2 yields
$$
\begin{aligned}
c 1 \leq &2(m-2)\|\mathcal{A}\|_{F}^{2}(\|x^{*}\|+\delta_{0})^{2 m-5}\|x-y\|+(m-3)\|b\|\|\mathcal{A}\|_{F}(\|x^{*}\|+\delta_{0})^{m-4}\|x-y\|, \\
c 2 \leq & 2(m-2)\|\mathcal{A}\|_{F}^{2}(\|x^{*}\|+\delta_{0})^{2 m-5}\|x-y\|.
\end{aligned}
$$
Direct manipulations lead to
$$
\begin{aligned}
\|\nabla^{2} f(x)-\nabla^{2} f(y)\| & \leq(m-1)(m-2) c 1+(m-1)^{2} c 2=L_{f}\|x-y\|,
\end{aligned}
$$
where $L_{f}$ is given in (\ref{equ:L_f}). This completes the proof.
\end{proof}

With the similar proof skills, we have the following property.

\begin{theorem}\label{thm:strongly smooth}
Suppose that $x^*$ is an optimal solution of (\ref{eq:model}). There exists $\delta _1>0$ such that for any $ x\in \mathcal{N} _s(x^*,\delta _1)$, $f(x)$ is strongly smooth, and its strong smoothness constant is given by
\begin{equation} \label{equ:M_2s}
\begin{aligned}
	M_{2s}=&(m-1)(2m-3)\| \mathcal{A} \| _{F}^{2}( \| x^* \|+\delta _1 ) ^{2m-4} \\
	&+(m-1)(m-2)\| b\| \| \mathcal{A} \| _F( \| x^* \| +\delta _1 ) ^{m-3}.
\end{aligned}
\end{equation}
\end{theorem}
%Note that the proof of Theorem \ref{thm:strongly smooth} is similar to that of Theorem \ref{thm:locally Hessian Lipschitz continuous}.

Next we discuss RSC of $f$ under the following assumption.

%\begin{proof}
%Similar to the proof in Theorem \ref{thm:locally Hessian Lipschitz continuous}.
%\end{proof}

\begin{assumption}\label{assumption:positive definite}
For any $T\in Q_{2s}( x^* ) $, $\nabla _{T}^{2}f( x^* )$ is positive definite.
\end{assumption}

Observe that problem \eqref{eq:model} reduces to the Compressed Sensing (CS) problem when $m=2$, and in this circumstance, Assumption \ref{assumption:positive definite} holds iff $\mathcal{A}$ is $2s$-regular \cite{}.
% (i.e., $\mathrm{A}_T$ is full column rank),
For illustration purpose, we give the following example in which Assumption \ref{assumption:positive definite} holds with $m>2$.

\begin{example}
  Let $u_{1}=((-1)^m,1, \cdots, 1)^{\top} \in \mathbb{R}^{n}$, $u_{2}=((-1)^{m-1},1, \cdots, 1)^{\top} \in \mathbb{R}^{n}$,  $\mathcal{A}=u_{1}^{m}+u_{2}^{m} \in S^{[m, n]}$, $b=u_1+(-1)^{m-1}u_2 \in \mathbb{R}^{n}$, $s=1$, where $m > 2$ is a given positive integer. It is easy to verify that $x^{*}=e_{1} \in \mathbb{R}^{n}$ is an optimal solution of problem \eqref{eq:model}, since
$$
\begin{aligned}
\mathcal{A}(x^{*})^{m-1} &=(u_{1}^{\top} x^{*})^{m-1} u_{1}+(u_{2}^{\top} x^{*})^{m-1}u_{2} =u_{1}+(-1)^{m-1}u_{2}=b,
\end{aligned}
$$
and $\|x^{*}\|_0=s=1$.
Now we claim that Assumption 1 holds at $x^{*}$ in two cases.
Note that
$$
\begin{aligned}
\mathcal{A}(x^{*})^{m-2} &=(u_{1}^{\top} x^{*})^{m-2} u_{1} u_{1}^{\top}+(u_{2}^{\top} x^{*})^{m-2} u_{2} u_{2}^{\top}=(-1)^{m}u_{1} u_{1}^{\top}+u_{2} u_{2}^{\top}.
\end{aligned}
$$
Case I: $m$ is a positive even integer. It follows that
$$\mathcal{A} ( x^* ) ^{m-2}=u_1u_{1}^{\top}+u_2u_{2}^{\top}=\left[ \begin{matrix}
	2&		0&		\cdots&		0\\
	0&		2&		\cdots&		2\\
	\vdots&		\vdots&		&		\vdots\\
	0&		2&		\cdots&		2\\
\end{matrix}\right ],$$
Thus, for any $T \in Q_{2 s}(x^{*})=\{ \{1\}, \{1, 2\},\{1, 3\} \cdots \{1, n\}\}$, we have
$$\begin{aligned}
	0\prec\nabla _{T}^{2}f(x^*)%&[(m-1)(m-2)(\mathcal{A} (x^*)^{m-3}(\mathcal{A} (x^*)^{m-1}-b))+(m-1)^2(\mathcal{A} (x^*)^{m-2})^2]_{T,T}\\
	=&\begin{cases}
	4\left( m-1 \right) ^2,           \qquad  \qquad   \qquad \;  \mbox{if~} \, T=\left\{ 1 \right\} ,\\
	(m-1)^2\left[ \begin{matrix}
	4&		0\\
	0&		4(n-1)\\
\end{matrix} \right] ,  \ \;  \mbox{if~} \, T\in \left\{ \left\{ 1,2 \right\} ,\cdots ,\left\{ 1,n \right\} \right\}.\\
\end{cases}\\
\end{aligned}$$
%which is positive definite.\\
Case II: $m$ is a positive odd integer. Note that
$$\mathcal{A} ( x^* ) ^{m-2}=-u_1u_{1}^{\top}+u_2u_{2}^{\top}=\left[ \begin{matrix}
	0&		2&		\cdots&		2\\
	2&		0&		\cdots&		0\\
	\vdots&		\vdots&		&		\vdots\\
	2&		0&		\cdots&		0\\
\end{matrix}\right ] .$$
Thus, for any $T \in Q_{2 s}(x^{*})=\{ \{1\}, \{1, 2\},\{1, 3\} \cdots \{1, n\}\}$,
$$  0\prec\nabla _{T}^{2}f(x^*)
    =\begin{cases}
	4\left( n-1 \right) \left( m-1 \right) ^2,   \qquad    \quad    \mbox{if~} \, T=\left\{ 1 \right\} ,\\
	(m-1)^2\left[ \begin{matrix}
	4(n-1)&		0\\
	0&		4\\
\end{matrix} \right] ,  \ \; \mbox{if~} \, T\in \left\{ \left\{ 1,2 \right\} ,\cdots ,\left\{ 1,n \right\} \right\}.\\
\end{cases}$$
\end{example}

Under Assumption \ref{assumption:positive definite}, the desired RSC is achieved as stated in the following theorem.

\begin{theorem}\label{thm:restricted strongly convex}
Suppose that $x^*$ is an optimal solution of (\ref{eq:model}) and Assumption \ref{assumption:positive definite} holds. There exists $\delta _0>\delta _2>0$ and $m_{2s}>0$ such that $\forall x\in \mathcal{N} _s(x^*,\delta _2)$, $f(x)$ is RSC, that is
$$m_{2s}\mathrm{I}\preceq \nabla _{T}^{2}f(x) ,  \quad \forall \ T\in Q_{2s}( x^* ) ,$$
where $\mathrm{I}$ is the identity matrix.
\end{theorem}

\begin{proof}
It is known from Theorem \ref{thm:locally Hessian Lipschitz continuous} that $\nabla^{2}f(x)$ is Lipschitz continuous for $x\in \mathcal{N}_{s}(x^*,\delta_0)$, thus for any $T\in Q_{2s}( x^* )$ and $x\in \mathcal{N} _{s}(x^*,\delta _2)$, $\nabla _{T}^{2}f(x)$ is Lipschitz continuous. Since $\nabla _{T}^{2}f(x^*)$ is positive definite, there exists $\delta _T>0$ such that $\forall x\in \mathcal{N} _s(x^*,\delta _T)$, $\nabla _{T}^{2}f(x)$ is also positive definite, and $m_T\mathrm{I}\preceq \nabla _{T}^{2}f(x)$ with $m_T=\min \{ \lambda _{\min}( \nabla _{T}^{2}f(x) ) |x\in \mathcal{N} _s(x^*,\delta _T) \}$. Set $\delta _2:=\min \{ \delta _0,\{ \delta _T \} _{T\in \mathcal{Q} _{2s}( x^* )} \} $ and $m_{2s}:=\min_{T\in \mathcal{Q} _{2s}( x^* )} \{ m_T \}$, therefore we have that $\nabla _{T}^{2}f(x)$ is positive definite and then $m_{2s}\mathrm{I}\preceq \nabla _{T}^{2}f(x)$ holds.
\end{proof}

\begin{definition}
 Given $\eta >0$ and $x^* \in \mathbb{R}^n$, we say $x^*$ is an $\eta$-stationary point of (\ref{eq:model}) if
$$
     x^*\in \mathcal{P} _s( x^*-\eta \nabla f( x^* ) ) ,
$$
 where $\mathcal{P} _s(x):=argmin_z\{\| x-z \|  :  \| z \| _0\le s\}$ is the sparse projection operator.
\end{definition}
% By using the definition of  $\mathcal{P} _s$, we can obtain the equivalent definition of the $\eta$-stationary point.
Using the definition of $\mathcal{P} _s$ and Lemma 2.2 in \citep{beck2013sparsity}, the equivalent definition of the $\eta$-stationary point is given in the following lemma.
\begin{lemma}\label{lemma:eta-stationary point}
The $s$-sparse vector $x^*$ is an $\eta$-stationary point of (\ref{eq:model}) if and only if
% \begin{equation}\label{equ:P_s}
% \begin{cases}
% 	\| \nabla _{( \Gamma ^* ) ^c}f( x^* ) \| _{\infty}\le | x^* |_{( s )}/\eta \,\,, \nabla _{\Gamma ^*}f( x^* ) =0,&		\,\,         \mathrm{if}\| x^* \| _0=s,\\
% 	\nabla f( x^* ) =0,&		\,\,         \mathrm{if}\| x^* \| _0<s.
% \end{cases}
% \end{equation}
% The specific form is as follows:
% $$\begin{cases}
% 	\mathcal{A} _{\cdot ,\cdot ,\Gamma ^*,\cdots ,\Gamma ^*}\left( x_{\Gamma ^*}^{*} \right) ^{m-2}\left( \mathcal{A} _{\cdot ,\Gamma ^*,\cdots ,\Gamma ^*}\left( x_{\Gamma ^*}^{*} \right) ^{m-1}-b \right) =0,&		\mathrm{if} \, \left\| x^* \right\| _0<s\\
% 	\mathcal{A} _{\Gamma ^*,\cdot ,\Gamma ^*,\cdots ,\Gamma ^*}\left( x_{\Gamma ^*}^{*} \right) ^{m-2}\left( \mathcal{A} _{\cdot ,\Gamma ^*,\cdots ,\Gamma ^*}\left( x_{\Gamma ^*}^{*} \right) ^{m-1}-b \right) =0,&		\mathrm{if} \, \left\| x^* \right\| _0=s\\
% 	\left\| \mathcal{A} _{\left( \Gamma ^* \right) ^c,\cdot ,\Gamma ^*,\cdots ,\Gamma ^*}\left( x_{\Gamma ^*}^{*} \right) ^{m-2}\left( \mathcal{A} _{\cdot ,\Gamma ^*,\cdots ,\Gamma ^*}\left( x_{\Gamma ^*}^{*} \right) ^{m-1}-b \right) \right\| _{\infty}\leqslant \left| x^* \right|_{(s)}/\eta .
% \end{cases} $$
$$\begin{cases}
	\mathcal{A} _{\cdot ,\cdot ,\Gamma ^*,\cdots ,\Gamma ^*}\left( x_{\Gamma ^*}^{*} \right) ^{m-2}\left( \mathcal{A} _{\cdot ,\Gamma ^*,\cdots ,\Gamma ^*}\left( x_{\Gamma ^*}^{*} \right) ^{m-1}-b \right) =0,     \qquad   \qquad      \qquad  \ \ \ \, if \, \left\| x^* \right\| _0<s,\\
	\begin{cases}
	\mathcal{A} _{\Gamma ^*,\cdot ,\Gamma ^*,\cdots ,\Gamma ^*}\left( x_{\Gamma ^*}^{*} \right) ^{m-2}\left( \mathcal{A} _{\cdot ,\Gamma ^*,\cdots ,\Gamma ^*}\left( x_{\Gamma ^*}^{*} \right) ^{m-1}-b \right) =0,\\
	\left\| \mathcal{A} _{\left( \Gamma ^* \right) ^c,\cdot ,\Gamma ^*,\cdots ,\Gamma ^*}\left( x_{\Gamma ^*}^{*} \right) ^{m-2}\left( \mathcal{A} _{\cdot ,\Gamma ^*,\cdots ,\Gamma ^*}\left( x_{\Gamma ^*}^{*} \right) ^{m-1}-b \right) \right\| _{\infty}\leqslant \frac{\left| x^* \right|_{(s)}}{\eta},\\
\end{cases} if \, \left\| x^* \right\| _0=s.\\
\end{cases}$$
\end{lemma}

The following first order necessary optimality condition of problem (\ref{eq:model}) can be obtained according to \citep[Theorem 2.2]{beck2013sparsity}.

\begin{theorem}\label{optimality condition}
Let $\eta <1/M_{2s}$, with $M_{2s}$ defined in Theorem \ref{thm:strongly smooth}. If $x^*$ is an optimal solution of (\ref{eq:model}), then the following fixed point equation holds, that is,
\begin{equation}\label{equ:fixed point equation}
  x^*=\mathcal{P} _s( x^*-\eta \nabla f( x^* ) ).
\end{equation}
\end{theorem}
%According to the optimality condition in Theorem \ref{optimality condition}, we can give a method to solve (\ref{eq:model}).
%In order to deal with the non-differentiability of $\mathcal{P} _s$, Zhou et al.\citep{zhou2021global} defined a system of differentiable nonlinear equations by using the special structure of sparse projection operator, which is given in the following definition. Following this, we gives an equivalent optimality condition of (\ref{equ:fixed point equation}) in Theorem \ref{thm:Equivalent characterization}.

In order to deal with the non-differentiability of $\mathcal{P} _s$, we follow the reformulation scheme from \citep{zhou2021global} and rewrite the optimality condition of (\ref{equ:fixed point equation}) as follows.

% Given $x\in \mathbb{R} ^n$ and $\eta >0$. Denote $u:=x-\eta \nabla f(x)$ and its best s-support index sets as
% $$\mathcal{T} (x;\eta ):=\{ T\subseteq [ n ] \; : \; |T|=s, | u_i |\ge | u_j |, \forall i\in T, \forall j\in T^c \}.$$
% For any given $T\in \mathcal{T} (x;\eta )$, the corresponding nonlinear equation is defined by
% \begin{equation}\label{equ:nonlinear equation}
% F_{\eta}(x;T):=\left[ \begin{array}{c}
%	\nabla _Tf(x)\\
%	x_{T^c}\\
%\end{array} \right] =\left[ \begin{array}{c}
%	( ( m-1 ) \mathcal{A} x^{m-2}( \mathcal{A} x^{m-1}-b ) ) _T\\
%	x_{T^c}\\
%\end{array} \right] =0.
% \end{equation}

\begin{theorem}\label{thm:Equivalent characterization}
 Given $\eta >0$, $x^*$ satisfies the fixed point equation \eqref{equ:fixed point equation} if and only if
\begin{equation}\label{equ:nonlinear equation}
 F_{\eta}(x^*;T):=\left[ \begin{array}{c}
	\nabla _Tf(x^*)\\
	x^*_{T^c}\\
\end{array} \right] =\left[ \begin{array}{c}
	( ( m-1 ) \mathcal{A} (x^*)^{m-2}( \mathcal{A} (x^*)^{m-1}-b ) ) _T\\
	x^*_{T^c}\\
\end{array} \right] =0,
 \end{equation}
for any index set $T\in \mathcal{T} (x^*;\eta )$, where $$\mathcal{T} (x^*;\eta ):=\{ T\subseteq [ n ] \; : \; |T|=s, | u^*_i |\ge | u^*_j |, \forall i\in T, \forall j\in T^c \}$$ with
$u^* := x^*-\eta \nabla f(x^*)$.
%$$F_{\eta}(x^*;T)=0,\ \ \forall T\in \mathcal{T} (x^*;\eta ).$$
Moreover, $\mathcal{T} (x^*;\eta )=\mathcal{J} _s( x^* ) $.
%\begin{proof}
%The proof of equivalence relation was given by Zhou et al.\citep{zhou2021global} in lemma 4, and the proof of $\mathcal{T} (x^*;\eta )=\mathcal{J} _s( x^* ) $ is similar to that of \citep[Theorem 3]{zhao2021lagrange}.
%\end{proof}
\end{theorem}
\begin{proof} The first part follows readily from \citep[Lemma 4]{zhou2021global}, and the ``moreover" part follows from \citep[Theorem 3]{zhao2021lagrange}.
\end{proof}

To solve the smooth equation system \eqref{equ:nonlinear equation} for a given index set $T\in \mathcal{T} (x^*;\eta )$, we next investigate the nonsingularity of the Jacobian matrix $\nabla F_{\eta}( x;T ) $ in a neighborhood of $x^*$, where
$$\nabla F_{\eta}( x;T ) =\left[ \begin{matrix}
	\nabla _{T}^{2}f( x )&		\nabla _{T,T_{}^{c}}^{2}f( x )\\
	0&		\mathrm{I}_{n-s}
\end{matrix} \right] .$$
Apparently, the nonsingularity of $\nabla F_{\eta}( x;T ) $ is equivalent to that of $\nabla _{T}^{2}f( x ) $. Thus, it suffices to show that $\nabla _{T}^{2}f( x ) $ is nonsingular for any $T\in \mathcal{T} (x;\eta )$ when $x$ is sufficiently close to $x^*$. Let $x^*$ be an optimal solution of (\ref{eq:model}). Set
$$\delta _3:=\min_{k\in \Gamma ^*} \{| x_{k}^{*} |-\eta \| \nabla _{( \Gamma ^* ) ^c}f( x^* ) \| _{\infty}\}/(\sqrt{2}( 1+\beta M_{2s} )).$$ It follows from Lemma \ref{lemma:eta-stationary point} that $\delta _3>0$.
\begin{lemma}\label{lemma:containment relationship}
Let $x^*$ be an optimal solution of (\ref{eq:model}) and $\delta ^*:=\min \{ \delta _1,\delta _2,\delta _3 \}$, then for any $x\in \mathcal{N} _s(x^*,\delta ^*)$, we have $$\mathcal{T} (x;\eta )\subseteq \mathcal{T} (x^*;\eta )\ \mathrm{and}\ \Gamma ^*\subseteq \mathrm{supp(}x)\cap T, \ \ \forall T\in \mathcal{T} (x;\eta ).$$
In particular, if $\| x^* \| _0=s$, $\{ \Gamma ^* \} =\{\mathrm{supp(}x)\}=\mathcal{T} (x;\eta )=\mathcal{T} ( x^*;\eta ) $.
\end{lemma}
\begin{proof}
Denote $\Gamma =\mathrm{supp(}x)$, $q=\nabla f( x ) $. For any $x\in \mathcal{N} _s(x^*,\delta ^*)$, we know that $\nabla f( x ) $ is Lipschitz continuous from the proof in Theorem \ref{thm:strongly smooth}, that is, $\| q-q^* \| \le M_{2s}\| x-x^* \| $ with $M_{2s}$ given in \eqref{equ:M_2s}. Together with Lemma 2 of \citep{zhao2021lagrange}, we can obtain the desired assertions.
\end{proof}

\begin{theorem}\label{thm:nonsingular}
Let $x^*$ be an optimal solution of (\ref{eq:model}) and Assumption \ref{assumption:positive definite} holds. Then for any $x\in \mathcal{N} _s(x^*,\delta ^*)$ and $T\in \mathcal{T} (x;\eta )$, $\nabla _{T}^{2}f( x )$ is positive definite.
\end{theorem}

\begin{proof}
According to Lemma \ref{lemma:containment relationship} and Theorem \ref{thm:Equivalent characterization}, we know for any $x\in \mathcal{N} _s(x^*,\delta ^*)$, $\mathcal{T} (x;\eta )\subseteq \mathcal{T} (x^*;\eta )=\mathcal{J} _s( x^* )\subseteq Q_{2s}(x^*)$. Consequently, Theorem \ref{thm:restricted strongly convex} implies that $\nabla _{T}^{2}f( x )$ is positive definite for any $x\in \mathcal{N} _s(x^*,\delta ^*)$ and $T\in \mathcal{T} (x;\eta )$.
\end{proof}

\subsection{NHTP algorithm}
In this subsection, we apply the Newton Hard-Threshold Pursuit (NHTP) algorithm  for problem \eqref{eq:model}, and analyze the locally quadratic convergence of the algorithm.

Given $\eta >0$, let $x^k$ be the current iteration point. NHTP firstly chooses one index set $T_k\in \mathcal{T} (x^k;\eta )$, and then does the Newton step for $F_{\eta}(x;T_k)$. Specifically, taking the following form for the nonlinear equation $F_{\eta}(x;T_k)=0$ to get the next iteration $\tilde{x}^{k+1}$:
$$\nabla F_{\eta}( x^k;T_k ) ( \tilde{x}^{k+1}-x^k ) =-F_{\eta}( x^k;T_k ).$$
Denote the Newton direction at the $k$-th iteration by $d_{N}^{k}:=\tilde{x}^{k+1}-x^k$. Substituting it into the above formula yields
\begin{equation}\label{equ:Newton direction}
\left \{ \begin{aligned}
	\nabla _{T_k}^{2}f( x^k ) ( d_{N}^{k} ) _{T_k}&=\nabla _{T_k,T_{k}^{c}}^{2}f( x^k ) x_{T_{k}^{c}}^{k}-\nabla _{T_k}f( x^k ),\\
	( d_{N}^{k} ) _{T_{k}^{c}}&=-x_{T_{k}^{c}}^{k}.
\end{aligned} \right.
\end{equation}
Note that $\nabla _{T_k}^{2}f( x^k ) $ is nonsingular under the condition in Theorem \ref{thm:nonsingular}.
Finally, the Armijo line search strategy is adopted to obtain the $(k+1)$th iteration $x^{k+1}=x^k(\alpha _k)$ and
$$x^k(\alpha _k):=\left[ \begin{array}{c}
	x_{T_k}^{k}+\alpha _k( d_{N}^{k} ) _{T_k}\\
	x_{T_{k}^{c}}^{k}+( d_{N}^{k} ) _{T_{k}^{c}}\\
\end{array}\right ] =\left[ \begin{array}{c}
	x_{T_k}^{k}+\alpha _k( d_{N}^{k} ) _{T_k}\\
	0\\
\end{array}\right] ,\quad \alpha _k>0,$$
where $\alpha _k$ is the step length. To measure the distance of the $k$th iteration from the $\eta$-stationary point, we take a accuracy measure as
$$\mathrm{Tol}_{\eta}(x^k;T_k):=\| F_{\eta}(x^k;T_k)\| +\max_{i\in T_k^c} \{\max\mathrm{(}|\nabla _if(x^k)|-| x^k |_{(s)}/\eta ,0)\}.$$
The framework of NHTP algorithm for solving (\ref{eq:model}) is summarized in Algorithm \ref{Alg1:NHTP}.
{\small \begin{algorithm}
\caption{NHTP for solving \eqref{eq:model}} \label{Alg1:NHTP}
{\bf Step 0} Initialize  $x^0$, choose $\eta, \gamma >0, \sigma \in (0,1/2), \beta \in (0,1)$ and set $k=0$.\protect{\\}
{\bf Step 1} Choose $T_k\in \mathcal{T} ( x^k;\eta )$. If $\mathrm{Tol}_{\eta}( x^k;T_k )=0$, then stop. Otherwise, go to Step 2.\protect{\\}
{\bf Step 2} Calculate the search direction by \eqref{equ:Newton direction}.\protect{\\}
{\bf Step 3} Compute $x^{k+1}=x^{k}(\alpha_{k})$, where $\alpha_{k}=\beta^{\ell}$ with $\ell$ the smallest integer such that
              $$ f(x^k(\beta ^{\ell}))\le f(x^k)+\sigma \beta ^{\ell}\langle \nabla f(x^k),d^k\rangle.$$
{\bf Step 4} Set $k=k+1$, and go to Step 1.

\end{algorithm}}
%\begin{enumerate}[leftmargin=0mm]
%\setlength{\itemsep}{0mm}
%  \item[{\bf Step 0}] Initialize  $x^0$, choose $\eta, \gamma >0, \sigma \in (0,1/2), \beta \in (0,1)$ and set $k=0$.
%  \item[{\bf Step 1}] Choose $T_k\in \mathcal{T} ( x^k;\eta )$. If $\mathrm{Tol}_{\eta}( x^k;T_k )=0$, then stop. Otherwise, go to Step 2.
%  \item[{\bf Step 2}] Calculate the search direction by \eqref{equ:Newton direction}.
%  \item[{\bf Step 3}] Compute $x^{k+1}=x^{k}(\alpha_{k})$, where $\alpha_{k}=\beta^{\ell}$ with $\ell$ the smallest integer such that
%              $$ f(x^k(\beta ^{\ell}))\le f(x^k)+\sigma \beta ^{\ell}\langle \nabla f(x^k),d^k\rangle.$$
%  \item[{\bf Step 4}] Set $k=k+1$, and go to Step 1.
%\end{enumerate}
% Next, for (\ref{eq:model}), combined with Zhou et al.\citep[Lemma 5]{zhou2021global}], in the neighborhood of $X^*$, we discuss the descent produced by $f( x ) $ iterating in the restricted subspace $x_{T_{k}^{c}}=0$ along the Newton direction $d_{N}^{k}$ defined by (\ref{equ:Newton direction}).

According to Theorem \ref{thm:strongly smooth} and Theorem \ref{thm:restricted strongly convex}, we know that $f$ is $m_{2s}$-RSS and $m_{2s}$-RSC for any $x\in\mathcal{N} _s(x^*,\delta ^*)$. We have the following decent lemma.
\begin{lemma}
Let $x^*$ be an optimal solution of (\ref{eq:model}). Given $\gamma \leqslant m_{2s}$ and $\eta \le 1/(4M_{2s})$, then in the neighborhood $\mathcal{N} _s(x^*,\delta ^*)$ of $x^*$, we have
$$\langle \nabla _{T_k}f(x^k),(d_{N}^{k})_{T_k}\rangle \le -\gamma \| d_{N}^{k} \| ^2+\frac{1}{4\eta}\| x_{T_{k}^{c}}^{k} \| ^2.$$
\end{lemma}
It can be seen that the restricted Newton direction $d_{N}^{k}$ provides a good descent direction on the restricted subspace $x_{T_{k}^{c}}=0$, which ensures the convergence of NHTP algorithm. We select the parameters $\gamma$, $\sigma$ and $\eta$ in the NHTP algorithm such that
\begin{equation}\label{equ:gamma, sigma and eta}
0<\gamma \le \min \{1,2M_{2s}\},\ 0<\sigma <1/2, \ \mathrm{and}\ 0<\beta <1.
\end{equation}
Furthermore, define the following two parameters as
\begin{equation}\label{equ:alpha and eta}
\bar{\alpha}:=\min \{ \frac{1-2\sigma}{M_{2s}/\gamma -\sigma},1 \} ,   \ \bar{\eta}:=\min \{ \frac{\gamma (\bar{\alpha}\beta )}{M_{2s}^{2}},\bar{\alpha}\beta ,\frac{1}{4M_{2s}} \} .
\end{equation}
% These parameters are crucial in algorithm design and convergence analysis.
% Combining Theorem \ref{thm:locally Hessian Lipschitz continuous}, Theorm\ref{thm:strongly smooth}, Theorem \ref{thm:restricted strongly convex}, and the choice of parameters in \eqref{equ:gamma, sigma and eta} and \eqref{equ:alpha and eta}, referring to Zhou et al. \citep[Theorem 10]{zhou2021global}, We analyze the locally quadratic convergence of the NHTP algorithm for optimization problem \eqref{eq:model}.

\begin{theorem}
Suppose the sequence $\{ x^k \} $ is generated by Algorithm \ref{Alg1:NHTP}, $x^*$ is an optimal solution of \eqref{eq:model} and Assumption \ref{assumption:positive definite} holds. Let the parameters $\gamma$, $\sigma$ and $\beta$ satisfy the conditions in \eqref{equ:gamma, sigma and eta}, $\bar{\eta}$ be defined in \eqref{equ:alpha and eta}, and $\eta \le \bar{\eta}$. If the initial point $x^0$ of the NHTP algorithm satisfies $x^0\in \mathcal{N} _s(x^*,\delta ^*)$, then we have
\begin{itemize}
    \item [(i)] $\lim_{k \rightarrow \infty} x_k=x^*$.
    \item [(ii)] The rate of convergence from $\{ x^k \} $ to $x^*$ is quadratic,
    $$\| x^{k+1}-x^* \| \le \frac{L_f}{2m_{2s}}\| x^k-x^* \| ^2.$$
\end{itemize}
\begin{proof}
It follows from Theorems \ref{thm:locally Hessian Lipschitz continuous}-\ref{thm:restricted strongly convex} that for any $x\in\mathcal{N} _s(x^*,\delta ^*)$, $f$ is restricted Hessian Lipschitz continuous, $M_{2s}$-RSS and $m_{2s}$-RSC. Thus the regularity conditions in \citep[Theorem 10]{zhou2021global} are satisfied, and then the locally quadratic convergence of the Algorithm \ref{Alg1:NHTP} can be established.
\end{proof}
\end{theorem}
\section{Numerical experiments}
\label{sec:Numerical Experiments}
In order to verify the effectiveness of the NHTP algorithm for \eqref{eq:model}, we conduct numerical experiments for the multilinear equations whose coefficient tensors are CP-tensors and symmetric strong M-tensors, and compare NHTP with the homotopy algorithm proposed in \citep{yan2022homotopy}.  All numerical examples are implemented on a laptop (2.40GHz, 16 GB of RAM) by using MATLAB (R2021a).
\vskip 2mm

In our experiments, the parameters in NHTP algorithm are chosen as: $\sigma =10^{-4}/2, \beta =0.5$, $\gamma =\gamma _k$ by updating $\gamma_{k}=10^{-10}$ if $x_{T_{i}}^{k}=0$, otherwise, $\gamma_{k}=10^{-4}$. Parameter $\eta$ is generated by: $\eta =\min\mathrm{(}\| x_{T}^{0}\| )/( 10(1+\max\mathrm{(}\| \nabla _{T^c}f(x^0)\| )))$ with $T=\mathrm{supp(}\mathcal{P} _s(x^0))$. Take $\mathrm{Tol}_{\eta _k}( x^k;T_k ) \le 10^{-7}$ as the stopping criterion. For the homotopy algorithm, set $\delta =0.75$, $\sigma =10^{-4}$, $\tau =10^{-5}$, $t_0=0.01$ and take $\| \mathcal{H} ( z^k ) \| <10^{-7}$ as its stopping criterion. We conduct 50 independent experiments for the following two examples, where $(m,n)$ takes different values, $s=\lceil 0.01n \rceil , \lceil 0.05n \rceil $.
\begin{example}
\textup{\textbf{(Random CP-tensors) }}
\label{example1}
Let $\mathcal{A} =\sum\nolimits_{s=1}^n{( u^{( s )} ) ^m}\in \mathrm{CP}^{[m,n]}$, where the components of $u^{(s)}( s\in [n] ) $ are randomly generated in $[0,1]$. The true values $x^*$, $b$ and the initial point $x^0$ are generated by the following Matlab pseudocode:
\end{example}

\begin{center}
\texttt{$x^*$=zeros($n$,1), $\Gamma$=randperm($n$), Tx=$\Gamma(1:s)$, $x^*$(Tx)=rand($s$,1),\\
$b$=$\mathcal{A}( x^*) ^{m-1}$, e=zeros($n$,1), e(Tx)=0.1*rand($s$,1), $x^0$=$x^*$+e.}
\end{center}

\begin{example}
 \textup{\textbf{(Random symmetric strong M-tensors generated by 0-1 uniform distribution)} }
 \label{example2}
 Let $\mathcal{A} =\mathrm{s}\mathcal{I} -\mathcal{B} \in \mathbb{R} ^{[ m,n ]}$, $\mathcal{B}$ be a symmetric tensor, where each element is randomly generated in the $[0,1]$, and $\mathrm{s}=n^{m-1}>\rho$, so $\mathcal{A}$ is a symmetric strong M-tensor. The true values $x^*$, $b$ and initial points $x^0$ are generated in the same way as in example \ref{example1}.
\end{example}

We collect the numerical results in Table \ref{table2} and Table \ref{table3}, where `(N$|$h)' denotes (NHTP algorithm$|$homotopy algorithm), `Re' denotes the average relative error, `Time' denotes the average CPU time,  `Iter' denotes the average number of iterations, and `nnz' denotes the number of non-zero elements of the solution in the average sense, obtained by the following formula:
$
\min \{t: \sum\nolimits_{i=1}^{t}|x|_{(i)} \geq 0.999\|x\|_{1}\}.
$

\begin{table}[h]
\tbl{Numerical results for Example \ref{example1}.}
{\begin{tabular}{ccccc}
\toprule
($m,n$) & nnz(N$|$h) & Re(N$|$h) & Time(s)(N$|$h) & Iter(N$|$h)\\
\midrule
(3,10) & \textbf{1}$|$9 &\textbf{7.25e-09}$|$6.06e-01	&0.011$|$0.004	&5$|$5\\
\midrule
\multirow{2}{*}{(3,30)} &\textbf{1}$|$27	&\textbf{5.49e-09}$|$2.69e-01	&0.014$|$0.006	&5$|$5\\
&\textbf{2}$|$22	&\textbf{1.82e-09}$|$2.59e-02	&0.017$|$0.005	&6$|$4\\
\midrule
\multirow{2}{*}{(3,50)} &\textbf{1}$|$43	&\textbf{8.86e-10}$|$2.42e-01	&0.016$|$0.008	&5$|$5\\
&\textbf{3}$|$41	&\textbf{9.94e-12}$|$4.69e-02	&0.020$|$0.006	&6$|$4\\
\midrule
\multirow{2}{*}{(3,70)}
&\textbf{1}$|$58	&\textbf{4.38e-11}$|$4.12e-01	&0.045$|$0.018	&5$|$5\\
&\textbf{4}$|$64	&\textbf{2.57e-11}$|$9.39e-02	&0.060$|$0.015	&7$|$4\\
\midrule
(4,10)	&\textbf{1}$|$10	&\textbf{2.14e-09}$|$2.33e-01	&0.017$|$0.011	&5$|$6\\
\midrule
\multirow{2}{*}{(4,30)}
&\textbf{1}$|$28	&\textbf{5.22e-10}$|$6.17e-01	&0.086$|$0.051	&5$|$6\\
&\textbf{2}$|$21	&\textbf{8.30e-09}$|$1.53e-01	&0.097$|$0.035	&6$|$5\\
\midrule
\multirow{2}{*}{(4,50)}
&\textbf{1}$|$47	&\textbf{3.19e-09}$|$8.48e-01	&0.476$|$0.272	&6$|$7\\
&\textbf{3}$|$48	&\textbf{9.77e-12}$|$1.75e-01	&0.602$|$0.185	&7$|$5\\
\bottomrule
\end{tabular}}
\label{table2}
\end{table}

\begin{table}[h]
\tbl{Numerical results for Example \ref{example2}.}
{\begin{tabular}{ccccc}
\toprule
($m,n$) & nnz(N$|$h) & Re(N$|$h) & Time(s)(N$|$h) & Iter(N$|$h)\\
\midrule
(3,10) &\textbf{1}$|$1	&\textbf{2.13e-10}$|$ 2.90e-04	&0.011$|$0.004	&5$|$4\\
\midrule
\multirow{2}{*}{(3,30)} &1$|$1	&\textbf{2.03e-13}$|$2.55e-09	&0.015$|$0.005	&5$|$4\\
&\textbf{2}$|$12	&\textbf{1.25e-14}$|$2.54e-03	&\textbf{0.016}$|$2.859	&6$|$208\\
\midrule
\multirow{2}{*}{(3,50)} &1$|$1	&\textbf{3.40e-11}$|$4.97e-05	&0.020$|$0.007	&6$|$5\\
&\textbf{3}$|$38	&\textbf{1.11e-14}$|$7.40e-03	&\textbf{0.021}$|$11.156	&6$|$703\\
\midrule
\multirow{2}{*}{(3,70)}
&1$|$1	&\textbf{3.21e-13}$|$1.56e-06	&0.049$|$0.015	&6$|$5\\
&\textbf{4}$|$65	&\textbf{2.17e-16}$|$9.41e-03	&\textbf{0.078}$|$33.561	&4$|$808\\
\midrule
(4,10)	&1$|$1	&\textbf{2.78e-12}$|$7.06e-04	&0.019$|$0.007	&6$|$5\\
\midrule
\multirow{2}{*}{(4,30)}
&1$|$1	&\textbf{5.16e-16}$|$5.74e-12	&0.132$|$0.036	&6$|$5\\
&\textbf{2}$|$9	&\textbf{1.43e-15}$|$2.84e-02	&\textbf{0.143}$|$39.010	&8$|$403\\
\midrule
\multirow{2}{*}{(4,50)}
&1$|$1	&\textbf{1.31e-17}$|$2.17e-11	&0.528$|$0.178	&6$|$5\\
&\textbf{3}$|$45	&\textbf{1.15e-17}$|$2.79e-02	&\textbf{0.690}$|$295.253	&8$|$504\\
\bottomrule
\end{tabular}}
\label{table3}
\end{table}

\vskip 2mm

As can be seen from Table \ref{table2}, for random CP-tensors, the NHTP algorithm reports solutions with high recovery accuracy in a very short CPU time for all testing instances. In addition, NHTP  has significant superiority in running time, accuracy of the reported solution and recovery of sparsity comparing with the homotopy algorithm. For symmetric strong M-tensors case. Table \ref{table3} illustrates similar results. Specifically, NHTP algorithm can obtain high-accuracy solutions in a very short CPU time, most of which are at level $10^{-10}$, and the sparsity recovery of the solutions performs well. However, the accuracy of the solution obtained by homotopy algorithm is about $10^{-2}$ for the case of $s=\lceil 0.05n \rceil $, and the sparsity recovery of the solution performs poor. In terms of running time, NHTP algorithm is significantly superiority than homotopy algorithm for the case of $s=\lceil 0.05n \rceil$.
\vskip 2mm

Overall, NHTP algorithm is an efficient and stable algorithm in solving the SLS optimization model for multilinear equations with CP-tensor and symmetric strong M-tensor comparing with the homotopy algorithm.
\section{Concluding remarks}\label{sec:Concluding remarks}
In this paper, a sparse least squares methods for multilinear equations has been modelled, in which the original $\ell_0$-norm constraint has been imposed to control the sparsity of the solutions. For the formulated SLS optimization problem, we have established its optimality condition, along with several regularity conditions of the objective function. Based on this, an NHTP algorithm has been developed for SLS solutions with locally quadratic convergence. Numerical experiments have verified the superiority of our method comparing with the existing homotopy algorithm.
\section*{Disclosure statement}
No potential conflict of interest was reported by the authors.
\section*{Funding}
This work has been supported by the Beijing Natural Science Foundation under Grant [number Z190002].

\end{document}